\documentclass[12pt]{article}
\usepackage{amsmath}
\usepackage{fullpage}
\usepackage{amsfonts,amssymb,latexsym,epsfig,amsthm}
\usepackage{epsfig}

\newtheorem{lemma}{Lemma}
\newtheorem{theorem}{Theorem}

\newtheorem{cor}{Corollary}

\newcommand{\Prob}{\mathbb{P}}

\newcommand{\K}{\mathcal{K}} 

\newcommand{\calK}{K}

\newcommand{\N}{\mathbb{N}}

\newcommand{\Z}{\mathbb{Z}}

\DeclareMathOperator{\tr}{tr}

\usepackage{color}
\definecolor{blue}{rgb}{0,0,1}
\definecolor{red}{rgb}{1,0,0}

\title{\bf Circular unitary ensemble with highly oscillatory potential}

\author{Jinho Baik\thanks{Department of Mathematics, University of Michigan, Ann Arbor, MI, 48109, USA \newline
email: \texttt{baik@umich.edu}}
}

\date{\today}
\begin{document}
\maketitle

\begin{abstract}
We study the effect of highly oscillatory potentials to the eigenvalues of a random matrix. 
Consider the circular unitary ensembles with an external potential which is periodic with the period comparable to the average spacing of the eigenvalues. 
We show that in this case  the density of states is periodic and does not converge in the large matrix limit, but the local correlation functions converge to some simple combinations of the sine kernel and the potential.
We evaluate the correlation functions exactly and also asymptotically. 
\end{abstract}

\section{Introduction and results}

The external potential of a random matrix ensemble changes the density of states in the large matrix limit, but it does not influence the local  correlations of the eigenvalues generically. 
The universal limiting local correlation functions are given in terms of the sine kernel for the unitary invariant ensembles. 
Here an assumption was made that  the external potentials are of macroscopic scale. 

In this paper we study what happens if the potential changes in the microscopic scaling comparable to the spacing of the eigenvalues. 
Then it is natural to imagine that the potential influences the eigenvalues in the microscopic level and hence the universality is disturbed. 
We study this in the setting of the circular unitary ensembles of random unitary matrices in the presence of highly oscillatory potentials. We  evaluate the local correlation functions exactly and also asymptotically, and study the influence of the potential in the local correlation level. 

For a real-valued potential $V(e^{i\theta})$, $\theta\in [0, 2\pi]$, set
\begin{equation}
   V_\Lambda(e^{i\theta})= V(e^{i\Lambda\theta}), \qquad  \Lambda\in \N. 
\end{equation}
Consider the unitary group $\mathcal{U}_N$ of unitary matrices of size $N$ and let the density function on $\mathcal{U}_N$ be proportional to $e^{-\tr V_\Lambda(U)}dU$. Here $dU$ is the
Haar measure on $\mathcal{U}_N$. 
Then the eigenvalue density function is given by 
\begin{equation}\label{eq:density}
   	p_{N, \Lambda}(e^{i\theta_1}, \cdots, e^{i\theta_{N}})
   = \frac1{Z_{N}} \prod_{1\le a<b\le N} |e^{i\theta_a}-e^{i\theta_b}|^2
   \prod_{a=1}^{N} e^{-V_\Lambda(e^{i\theta_a})},
\end{equation}
where $Z_{N}$ is the normalization constant. 
Note that throughout this paper
\begin{itemize}
\item $N$ is the number of eigenvalues, and
\item $\frac{2\pi}{\Lambda}$ is the period of the potential. 
\end{itemize}
We also define the weights
\begin{equation}
   w(e^{i\theta})= e^{-V(e^{i\theta})}, \qquad w_\Lambda(e^{i\theta})=e^{-V_\Lambda(e^{i\theta})}
\end{equation}
for convenience. 

As an example, if $V(e^{i\theta})= -t\cos(\theta)$, $t>0$, and $\Lambda=N$, then $V_N(e^{i\theta})= -t\cos(N\theta)$ has $N$ local minima at
$\theta= \frac{2\pi k}{N}$, $k=0,1,\dots, N-1$. 
These $N$ local minima give an extra repulsion between the eigenvalues in addition to the 
log-repulsions due to the Vandermonde term in~\eqref{eq:density} in the microscopic scaling.

The standard circular unitary ensemble (CUE) corresponds to the case when $V(e^{i\theta})=1$. In this case, the density of states is uniform and the local correlation functions of the eigenvalues are given by the sine kernel. 
If we consider the eigenvalue density with macroscopically-varying potential $W(e^{i\theta})$ given as 
\begin{equation}\label{eq:density2}
   p_{N}(e^{i\theta_1}, \cdots, e^{i\theta_{N}})
   = \frac1{Z_{N}} \prod_{1\le a<b\le N} |e^{i\theta_a}-e^{i\theta_b}|^2
   \prod_{a=1}^{N} e^{-NW(e^{i\theta_a})},
\end{equation}
then the density of states depends on $W$, but local correlations given in terms of the sine kernel is still universal \cite{PasturShch1997, BleherIts, DeiftKMVZ99}. See also \cite{ABD} and references in it. 
Note the presence of the factor $N$ in front of the potential $W$. 
With this $N$, the density of states changes from the uniform distribution. If $N$ is not present, the density of states is still uniform.

Since the eigenvalue density function~\eqref{eq:density} is invariant under the shift $\theta_j\mapsto \theta_j+\frac{2\pi}{\Lambda}$, the density of states is periodic with period $\theta=\frac{2\pi}{\Lambda}$ and hence it does not converge. 
However, we will see below that it is still bounded as $N\to \infty$, $\Lambda=O(N)$. 
On the other hand, if we replace the term in the last product in~\eqref{eq:density} to $e^{-NV_\Lambda(e^{i\theta})}$, then the density of states is not bounded any more. See Section~\ref{sec:example}. 

\medskip

The density~\eqref{eq:density} was  previously studied by Forrester \cite{Forrester}. He considered the case when $\Lambda=N$ and evaluated the partition function and  the one-point and two-point correlation functions exactly and asymptotically. 
In this paper, we extend Forrester's result and evaluate all $m$-point correlation functions when $\Lambda/N$ is rational. 
(We note, however, that the case when $\Lambda/N$ is an integer can be obtained easily from the case when $\Lambda=N$ by renaming the potential.)
In \cite{Forrester}, the integrals were evaluated explicitly using the Fourier expansion of the weight. 
In this paper we use a different method and use a simple relation between orthogonal polynomials  with respect to $e^{-V_\Lambda(e^{i\theta})}$ and those with respect to $e^{-V(e^{i\theta})}$. 
Forrester's work was motivated by the Kosterlitz-Thouless conducting-insulating phase transition. 
A different motivation which inspires this paper is discussed in Section~\ref{sec:motivation} below. 

Forrester \cite{Forrester} also evaluated the free energy and  the one-point and two-point correlation functions for the cases when the term $|e^{i\theta_a}-e^{i\theta_b}|^2$ in~\eqref{eq:density} is replaced by 
$|e^{i\theta_a}-e^{i\theta_b}|^\beta$ for $\beta=1$ and $\beta=4$. These cases correspond to orthogonal and symplectic invariant ensembles. 
It is an interesting open question to evaluate the higher correlation functions for these cases.

\subsection{Orthogonal polynomials}

The correlation functions for unitary ensembles are expressed in terms of orthogonal polynomials with respect to the weight
\cite{Mehta, TW2}. 
We first state a simple observation about the orthogonal polynomials with respect to periodic weights on the unit circle. 

Let $\pi_\ell^{(\Lambda)}(z) = z^\ell+\cdots$ be the monic orthogonal polynomial of
degree $\ell$ on the unit circle 
with respect to the weight $w_\Lambda$.
It is determined by the conditions 
\begin{equation}\label{eq:OPcond}
  \int_{0}^{2\pi} e^{-im\theta} \pi_\ell^{(\Lambda)}(e^{i\theta}) w_\Lambda(e^{i\theta}) d\theta =0, \qquad
  m=0,1,\dots, \ell-1.
\end{equation}
In addition, set
\begin{equation}\label{eq:kap}
  \kappa_\ell^{(\Lambda)} = \bigg\{ \int_{0}^{2\pi} |\pi_\ell^{(\Lambda)}(e^{i\theta})|^2 w_\Lambda(e^{i\theta}) d\theta \bigg\}^{-1/2}.
\end{equation}
We assume that the orthogonal polynomials exist for all $\ell$ and $\Lambda$. 
This holds, for example, if we assume that $V$ is piecewise continuous. 

The main observation is that there are simple relations between the orthogonal polynomials with respect to $w_{\Lambda_1}$ and the orthogonal polynomials with respect to $w_{\Lambda_2}$ if $\Lambda_1$ is an integer multiple of $\Lambda_2$.

\begin{lemma}\label{lemma0}
For $n, \lambda, L\in \mathbb{N}$ and $k=0, 1, \cdots, L-1$, 
\begin{equation}
    \pi_{n L+k}^{(\lambda L)}(z)= z^k\pi_n^{(\lambda)}(z^L)
\end{equation}
and $\kappa_{n L+k}^{(\lambda L)}=\kappa_n^{(\lambda)}$.
\end{lemma}

\begin{proof}
Clearly $z^k\pi_n^{(\lambda)}(z^L)$ is a monic polynomial of degree $nL+k$. We
will check that it satisfies the orthogonality conditions 
for $\pi_{nL+k}^{(\lambda L)}$. 
For $m\in \Z$, 
\begin{equation}\label{eq:intco}
  \int_{0}^{2\pi} e^{-im\theta} e^{ik\theta} \pi_{n}^{(\lambda)}(e^{iL\theta}) w_{\lambda L}(e^{i\theta}) d\theta
  = \frac1{L} \int_{0}^{2\pi L} e^{-i\frac{m-k}{L}\phi} \pi_{n}^{(\lambda)}(e^{i\phi}) w_{\lambda}(e^{i\phi}) d\phi,
\end{equation}
where we used the change of variables $L\theta=\phi$,
Writing the integral as the sum of $L$ integrals over $\phi\in [2\pi j, 2\pi(j+1))$, $j=0,1,\dots, L-1$, and then shifting the argument by $\phi \mapsto \phi+2\pi j$ in each integral, we find that~\eqref{eq:intco} equals
\begin{equation}\label{eq:intco2}
  	\frac1{L}(1+r+\cdots + r^{L-1}) \int_{0}^{2\pi} e^{-i\frac{m-k}{L}\phi} \pi_{n}^{(\lambda)}(e^{i\phi}) w_\lambda(e^{i\phi}) d\phi, 
	\qquad r:= e^{-2\pi i\frac{m-k}{L}}.
\end{equation}
Now for $m-k\notin L\Z$, we have $r\neq 1$ but $r^L=1$, and hence the sum $1+r+\cdots + r^{L-1}=0$. 
On the other hand, for $m-k\in L\Z$, if $q=\frac{m-k}{L}\in \Z$, 
 the integral $\int_{0}^{2\pi} e^{-iq\phi} \pi_{n}^{(\lambda)}(e^{i\phi}) w_{\lambda}(e^{i\phi}) d\phi=0$ for $q=0,1, \cdots, n-1$ 
due to the orthogonality conditions of $\pi_{n}^{(\lambda)}$. 
Therefore, since $0\le k<L$, we find that~\eqref{eq:intco2} equals $0$ for all $m=0,1,\cdots, nL+k-1$. 
This proves that $z^k\pi_n^{(\lambda)}(z^L)=\pi_{nL+k}^{(\lambda L)}(z)$. 
From this it follow that $\kappa_{n L+k}^{(\lambda L)}=\kappa_n^{(\lambda)}$.
\end{proof}

We also note a simple relation
\begin{equation}\label{eq:piLkz}
    e^{\frac{-2\pi i\ell k}\Lambda} \pi_{\ell}^{(\Lambda)}(ze^{\frac{2\pi ik}\Lambda})= \pi_{\ell}^{(\Lambda)}(z), 
\end{equation}
for all integers $k$, which is simple to check from the conditions~\eqref{eq:OPcond}.

\subsection{Correlation functions}

From the Gaudin-Mehta method for circular unitary ensembles \cite{Mehta, TW2}, the correlation functions are expressible in terms of the reproducing kernel 
\begin{equation}\label{eq:KCD}
  	\K_{N,\Lambda}(e^{i\theta}, e^{i\varphi})
	= \sum_{\ell=0}^{N-1} (\kappa_{\ell}^{(\Lambda)})^2 \overline{ \pi_{\ell}^{(\Lambda)}(e^{i\theta})} \pi_{\ell}^{(\Lambda)}(e^{i\varphi}) 
	\sqrt{w_\Lambda(e^{i\theta})w_\Lambda(e^{i\varphi})}.
\end{equation}
The Christoffel-Darboux formula (see Theorem 11.4.2 of \cite{Szego}) simplifies this sum. 
The conjugated kernel 
$\calK_{N,\Lambda}(e^{i\theta}, e^{i\varphi})= e^{i\frac{(N-1)}2\theta} \K_{N,\Lambda}(e^{i\theta}, e^{i\varphi}) e^{-i \frac{(N-1)}2\varphi}$ then becomes 
\begin{equation}\label{eq:calK}
\begin{split}
   \calK_{N,\Lambda}(e^{i\theta}, e^{i\varphi})
   &=   \frac{ e^{-i\frac{N}2(\theta-\varphi)} \Psi_{N}^{(\Lambda)}(e^{i\theta})\overline{\Psi_{N}^{(\Lambda)}(e^{i\varphi})}
   - e^{i\frac{N}2(\theta-\varphi)}
     \overline{\Psi_{N}^{(\Lambda)}(e^{i\theta})} \Psi_{N}^{(\Lambda)}(e^{i\varphi})}{2i \sin(\frac12(\theta-\varphi))}.
\end{split}
\end{equation}
where 
\begin{equation}
\begin{split}
   \Psi_{N}^{(\Lambda)}(e^{i\theta}):= \kappa_N^{(\Lambda)} \pi_{N}^{(\Lambda)}(e^{i\theta}) \sqrt{w_\Lambda(e^{i\theta})}
   .
\end{split}
\end{equation}
We remark that~\eqref{eq:piLkz} implies the natural  symmetry  
\begin{equation}\label{eq:calK001}
\begin{split}
   \calK_{N,\Lambda}(e^{i(\theta+ \frac{2\pi k}{\Lambda})}, e^{i(\varphi+ \frac{2\pi k}{\Lambda})})
	=\calK_{N,\Lambda}(e^{i\theta}, e^{i\varphi})
\end{split}
\end{equation}
for all integers $k$. 

Lemma~\ref{lemma0} implies that $\Psi_{nL}^{(\lambda L)}(e^{i\frac{x}L})= \Psi_{n}^{(\lambda)}(e^{ix})$. 
From this we easily find the following. 

\begin{theorem}\label{thm1}
For positive integers $n, \lambda, L$, 
\begin{equation}\label{eq:thm1}
\begin{split}
  \calK_{nL, \lambda L}(e^{i\frac{\xi}{L}}, e^{i\frac{\eta}{L} })
  &= \frac{\sin(\frac12(\xi-\eta))}{\sin(\frac{1}{2L}(\xi-\eta))} \calK_{n, \lambda}(e^{i  \xi}, e^{i  \eta}).
\end{split}
\end{equation}
\end{theorem}

When $n=1$, the right-hand side is 
particularly simple. 

\begin{cor}\label{cor:simn1}
We have 
\begin{equation}\label{eq:corsim}
\begin{split}
   \calK_{L, \lambda L}(e^{i\frac{\xi}{L}}, e^{i\frac{\eta}{L}})
   &=   \frac{\sin(\frac12(\xi-\eta))}{  \sin(\frac{1}{2L}(\xi-\eta))}  \frac{\sqrt{w(e^{ i\lambda \xi})w(e^{ i\lambda \eta})}}{w_0}, 
  \qquad w_0:=  \int_0^{2\pi} w(e^{i\theta}) d\theta.
\end{split}
\end{equation}
\end{cor}

\begin{proof}
We have $\pi_0^{(\lambda)}(z)=1$ and $\kappa_0^{(\lambda)}= \big( \int_0^{2\pi} w(e^{i\lambda \theta})d\theta \big)^{-1/2}= w_0^{-1/2}$. 
From~\eqref{eq:KCD}, $\K_{1, \lambda}(e^{i\theta}, e^{i\varphi})= \frac{\sqrt{w(e^{i\lambda \theta})w(e^{i\lambda\theta})}}{w_0}$.
Since $\calK_{1, \lambda}= \K_{1,\lambda}$, we obtain the result. 
\end{proof}

\bigskip

The density of state $\rho_{N,\Lambda}$ is related to $K_{N, \Lambda}$ by the formula
\begin{equation}
	\rho_{N,\Lambda}(e^{i\theta})= \frac1{N} K_{N, \Lambda}(e^{i\theta}, e^{i\theta}).
\end{equation}
Note that since the eigenvalue density function is invariant under the shift $\theta_j\mapsto \theta_j+\frac{2\pi}{\Lambda}$, the density of states is periodic: 
\begin{equation}\label{eq:peridos}
	\rho_{N,\Lambda}(e^{i\theta})=\rho_{N,\Lambda}(e^{i(\theta+\frac{2\pi}{\Lambda})}).
\end{equation}
Theorem~\ref{thm1} implies the following. 
\begin{theorem}
We have 
\begin{equation}\label{eq:dos3}
\begin{split}
   \rho_{nL, \lambda L}(e^{i \frac{\xi}{L} })
  &= \frac{1}{n}  \calK_{n, \lambda}(e^{i\xi}, e^{i\xi}).
\end{split}
\end{equation} 
When $n=1$, this can be simplified to 
\begin{equation}\label{eq:dos2}
\begin{split}
  \rho_{L, \lambda L}(e^{i \frac{\xi}{L} })
  &= \frac{w(e^{i\lambda \xi})}{w_0}. 
\end{split}
\end{equation} 
\end{theorem}

The formula~\eqref{eq:dos2} when $\lambda=1$ was first obtained in \cite{Forrester}. From this~\eqref{eq:dos2} for integer $\lambda$ also follows immediately by renaming the potential. 

Due to the periodicity, the density of states does not converge as $L\to\infty$. 
The above theorem shows that  it is nevertheless bounded. 

A comparison with the ensembles with macrocopically-varying potentials is in order. Note that in the eigenvalue density~\eqref{eq:density2}, there is the factor $N$ in front of the potential $W$. In this case, the density of states converges to the equilibrium measure $\rho_W^{eq}$ associated to $W$ as $N\to \infty$ for a very general class of potentials.
Suppose now we put the factor $N$ in front of a highly oscillatory potentials, namely 
\begin{equation}\label{eq:density3}
   \frac1{Z_{N}} \prod_{1\le a<b\le N} |e^{i\theta_a}-e^{i\theta_b}|^2
   \prod_{a=1}^{N} e^{-N V_\Lambda(e^{i\theta_a})}.
\end{equation}
Then~\eqref{eq:dos2} becomes $e^{-NV(e^{i\lambda \xi})}$ divided by $\int_0^{2\pi} e^{-NV(e^{i\lambda \xi})} d\xi$.
If $V(e^{i\theta})$ has one global minimum at $\theta_0$ in $[-\pi, \pi]$, then the Laplace's method shows that as $L\to \infty$ the density of states converges to $0$ for all $\xi \notin \frac{1}{\lambda}(\theta_0+2\pi \Z)$ and tends to $\infty$ at $\xi\in \frac{1}{\lambda}(\theta_0+2\pi \Z)$. 
Hence the density of states is not bounded in this case. 
See Section~\ref{sec:example} for more discussions. 

\bigskip

Let $R_{m}^{N,\Lambda}(e^{i\theta_1}, \cdots, e^{i\theta_{m}})$ denote the $m$-point correlation function of~\eqref{eq:density}. 
Note that the correlation functions are invariant under the shift $\theta_j\mapsto \theta_j+\frac{2\pi}{\Lambda}$ since so is the eigenvalue density function.  
The Mehta-Gaudin method implies that 
the $m$-point correlation function equals $\det(K_{N,\Lambda}(e^{i\theta_a}, e^{i\theta_a}))_{a,b=1}^m$. 
The above result for the reproducing kernel immediately yields the following 
\begin{theorem}\label{thm:corr}
Let $m$ be a positive integer. 
We have 
\begin{equation}\label{eq:Rlimit1}
\begin{split}
   \lim_{L\to \infty} 
   \bigg(\frac{1}{L}\bigg)^{m}  R_{m}^{nL, \lambda L}(e^{i\frac{\xi_1}{L} }, \cdots, e^{i\frac{\xi_m}{L} })
  &= \det\bigg[ \frac{\sin(\frac{1}{2}(\xi_a-\xi_b))}{\frac1{2}(\xi_a-\xi_b)} \calK_{n, \lambda}(e^{i\xi_a }, e^{i\xi_b }) \bigg]_{a,b=1}^{m}
\end{split}
\end{equation}
for $\xi$ and $\eta$ in a compact set. 
When $n=1$, this can be simplified to 
\begin{equation}\label{eq:Rlimit2}
\begin{split}
   \lim_{L\to \infty} 
   \bigg(\frac{1}{L}\bigg)^{m}  R_{m}^{L, \lambda L}(e^{i\frac{\xi_1}{L} }, \cdots, e^{i\frac{\xi_m}{L} })
  &= \frac{w(e^{i\lambda \xi_1})\cdots w(e^{i \lambda \xi_m})}{w_0^{m}} \det\bigg[ \frac{\sin(\frac12(\xi_a-\xi_b))}{\frac12(\xi_a-\xi_b)} \bigg]_{a,b=1}^{m}.
\end{split}
\end{equation}
\end{theorem}

Hence the limiting local correlations functions are simple combinations of the sine kernel and the potential. 
See Figure~\ref{fig:twocor} and~\ref{fig:twocor2} for some plots. 

The result~\eqref{eq:Rlimit2} for $m=1$ or $m=2$ with $\lambda=1$ was first obtained in \cite{Forrester}. 

We note that the results of this section are obtained under the minimal condition on the potential $V$ that the orthogonal polynomials $\pi_{\ell}^{\Lambda}$ exist for all $\ell$ and $\Lambda$. This holds, for example, if $V$ is piecewise continuous. 

\begin{figure}[htbp]
  \centering
  \includegraphics[scale=0.3]{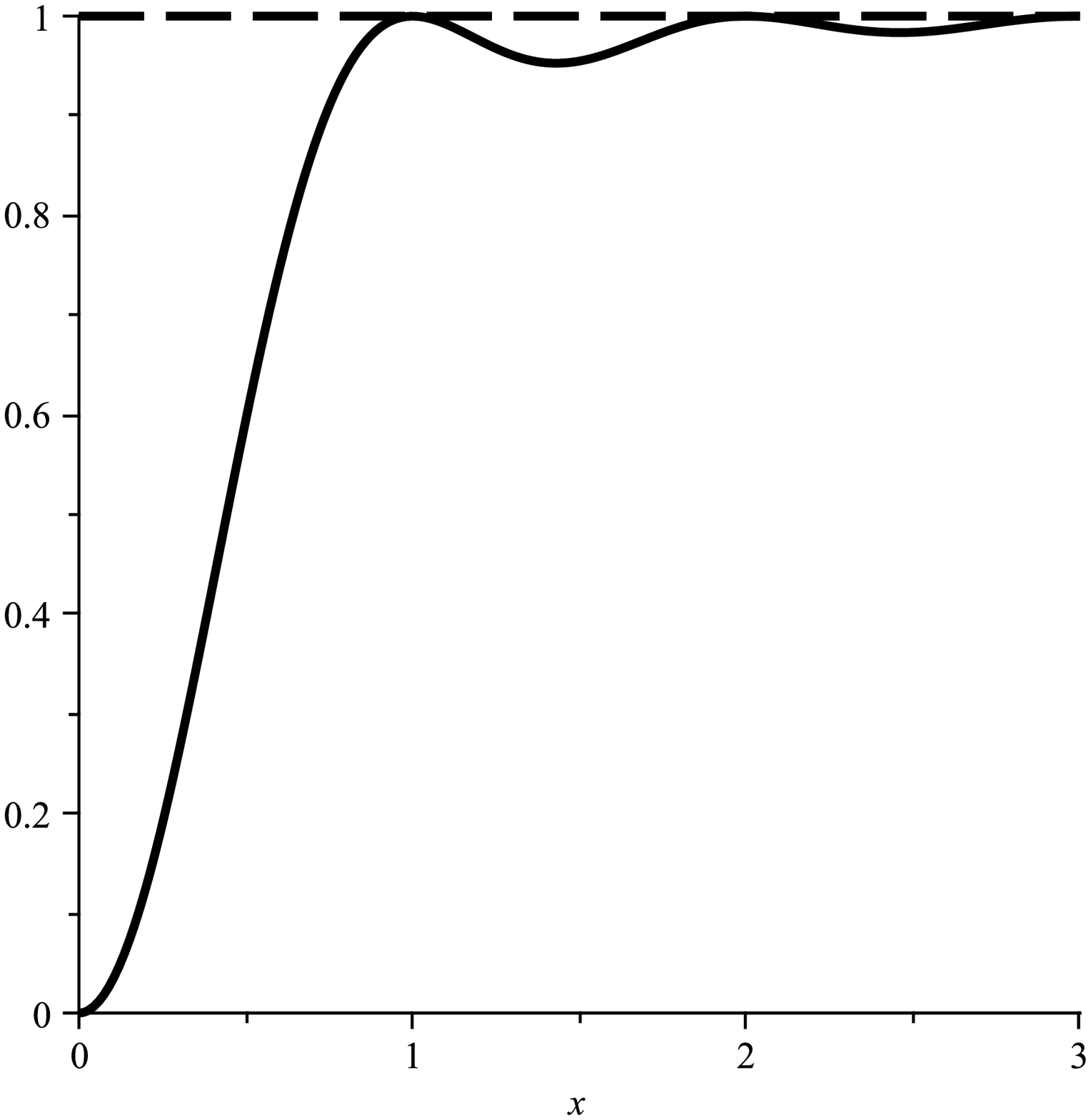}
  \includegraphics[scale=0.3]{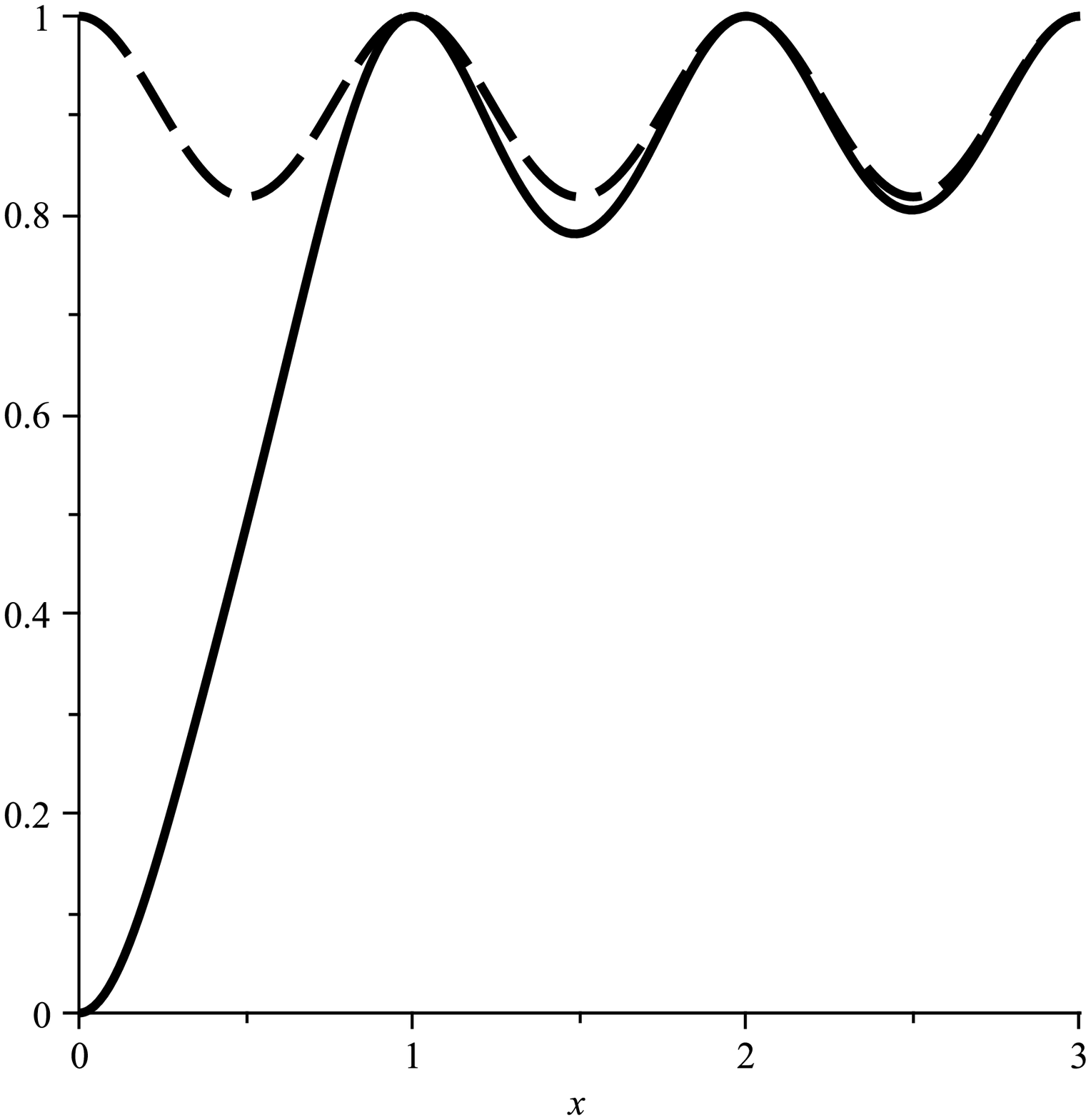}
  \caption{The dashed lines represent the density of states and the solid lines represent the limiting $2$-point function at $\xi_1=0$ and $\xi_2=2\pi x$. The heights are scaled so that the density of states is $1$ at $x=0$. 
Left: This is for the standard circular ensemble. The solid line is the graph of $R_2(0, x)=1- (\frac{\sin(\pi x)}{\pi x})^2$. 
Right: This is for the potential $V(e^{i\theta})=-\frac1{10} \cos(\theta)$ when $n=\lambda=1$. 
The solid  line is proportional to $\frac{w(1) w(e^{2\pi ix})}{w_0^2} (1- (\frac{\sin(\pi x)}{\pi x})^2)$ and the dashed line is proportional to $\frac{w(e^{2\pi ix})}{w_0}$ where $w(e^{2\pi ix})
= e^{\frac1{10} \cos(2\pi x)}$.}
  \label{fig:twocor}
\end{figure}

\begin{figure}[htbp]
  \centering
  \includegraphics[scale=0.3]{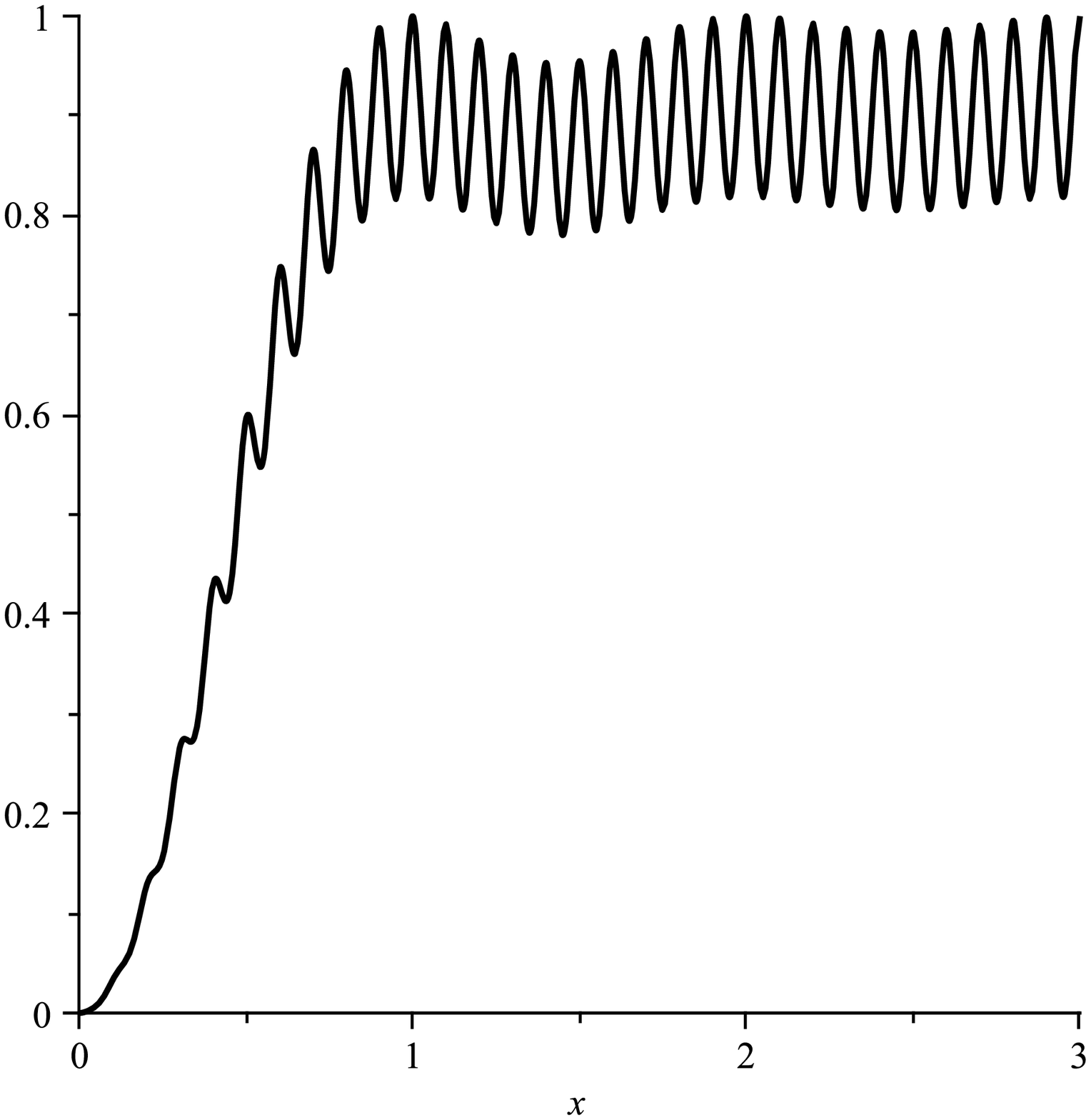}
  \includegraphics[scale=0.3]{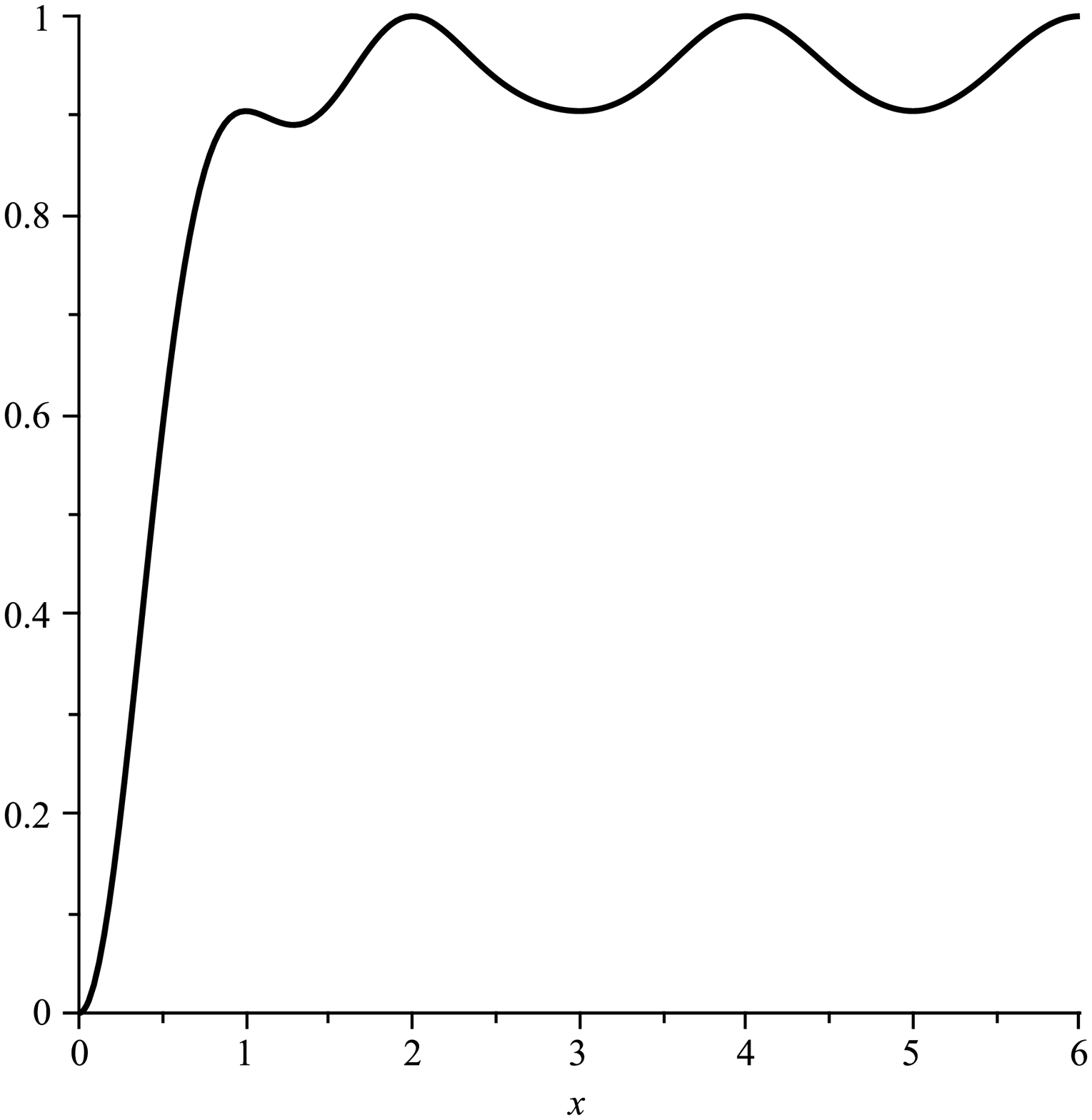}
  \caption{These are the graphs of the height-adjusted limiting $2$-point correlation functions with same potential $V(e^{i\theta})=-\frac1{10} \cos(\theta)$ as in Figure~\ref{fig:twocor}. Note that the $x$-axis on the right graph is scaled differently. Left: This is the case when $n=1$ and $\lambda=10$. There are $10$ times more ``wells'' of the oscillatory potential than the number of eigenvalues. 
Right: This is the case when $n=2$ and $\lambda=1$. There are twice more eigenvalues than the wells. 
}
  \label{fig:twocor2}
\end{figure}

\subsection{When $\Lambda=o(N)$:}

In Theorem~\ref{thm:corr}, we assumed that $\Lambda$ and $N$ are of same order.
We now consider what happens when $\Lambda=o(N)$. 
This corresponds to the case when $\lambda=1$ and $n\to \infty$ in the notations of the previous section.

Note that the right-hand side of~\eqref{eq:dos3} is the density of states for the ensemble with weight $e^{-V(e^{i\xi})}$ when $\lambda=1$. Hence if we replace $V$ by $nW$, this converges to the density function of the equilibrium measure $\rho_W^{eq}$ for $W$ as $n\to \infty$. 
This leads us to consider the following eigenvalue density function: 
\begin{equation}
\begin{split}
   \frac1{Z_{N}} \prod_{1\le a<b\le N} |e^{i\theta_a}-e^{i\theta_b}|^2
   \prod_{a=1}^{N} 	e^{-\frac{N}{\Lambda} W(e^{i\Lambda \theta_a})}.
\end{split}
\end{equation} 
Note the factor $\frac{N}{\Lambda}$ in front of the potential $W$. This factor is order $N$ when $\Lambda$ is order $1$, and is of order $1$ when $\Lambda$ is of same order as $N$. 
We assume that $W$ is a potential such that if we consider the  ensemble~\eqref{eq:density2}, then the density of states converges to the density function of the equilibrium measure $\rho_W^{eq}$ and the local statistics converge to the sine kernel in the bulk. 
We then obtain: 
\begin{theorem}
As $N\to \infty$ with $\Lambda = o(N)$ and $\frac{N}{\Lambda}\in \N$, the density of states satisfies 
\begin{equation}\label{eq:lth1}
\begin{split}
   	\rho_{N, \Lambda}(e^{i\frac{x}{\Lambda}}) \to \rho_W^{eq}(e^{ix}).
\end{split}
\end{equation} 
The $m$th correlation function satisfies 
\begin{equation}\label{eq:lth2}
\begin{split}
   	\bigg(\frac{1}{\rho(x) N}\bigg)^{m}  R_{m}^{N, \Lambda}(e^{i(\frac{x}{\Lambda}+\frac{\xi_1}{\rho(x) N}) }, \cdots, e^{i(\frac{x}{\Lambda}+\frac{\xi_m}{\rho(x) N})  })
  &\to  \det\bigg[ \frac{\sin(\pi(\xi_a-\xi_b))}{\pi(\xi_a-\xi_b)} \bigg]_{a,b=1}^{m}
\end{split}
\end{equation} 
for $x$ satisfying  $\rho(x):=\rho_{W}^{eq}(e^{ix})>0$. 
\end{theorem}

\begin{proof}
The result~\eqref{eq:lth1} is already discussed. For~\eqref{eq:lth2}, note that~\eqref{eq:thm1} implies  
\begin{equation}
\begin{split}
   	\frac1{\rho(x) N} \calK_{N, \Lambda} (e^{i(\frac{x}{\Lambda}+\frac{\xi}{\rho(x) N})}, e^{i(\frac{x}{\Lambda}+\frac{\eta}{\rho(x) N})})
	= \frac{\sin(\frac{\Lambda}{2N}(\xi-\eta))}{\Lambda \sin(\frac1{2N}(\xi-\eta))} 
	\frac{\Lambda}{\rho(x)N} \calK_{\frac{N}{\Lambda}, 1} (e^{i(x+\frac{\Lambda}{\rho(x) N}\xi}), e^{i(x+\frac{\Lambda}{\rho(x) N}\eta)})
\end{split}
\end{equation} 
assuming $\rho(x)>0$. 
The fraction involving the sine functions tends to $1$ in the limit. On the other hand, the rest converges to the sine kernel as $\frac{N}{\Lambda} \to \infty$. 
Hence we obtain the result. 
\end{proof}

\subsection{A motivation}\label{sec:motivation}

In \cite{RawalRodgers}, the authors measured the bumper-to-bumper distances between cars parked parallel to the curb on London streets. The streets were without any driveways or side streets and the data were collected in the late evening so that there were few spaces capable of taking additional cars. 
The authors of \cite{RawalRodgers} found that their measurement did not agree with the standard random sequential adsorption also known as random car parking models \cite{Renyi} (see \cite{EvansJ, Cadilhe} for review).
In particular the data show that the density tends to zero as the spacing approaches zero whereas
the mean gap density for the random sequential adsorption diverges like a logarithmic function in the same limit \cite{MacKenzie,Gonzalez,Krapivisky}. 
Subsequently it was pointed out in \cite{AbulMagd} that the spacing distribution empirically obtained in \cite{RawalRodgers} is close to the Wigner surmise for Gaussian unitary ensemble, which is an approximation to the limiting spacing distribution of the eigenvalues given by the Gaudin distribution. 
A different model using Markov processes was also proposed in \cite{SebaCar} to explain the data of \cite{RawalRodgers} and also the data collected in the streets of Hradec Kr\'alov\'e, a city in the Czech Republic. 

The above data were collected on the streets free of parking meters. But many streets have parking meters. 
And if we consider parked cars in a parking garage, there are marking on the floor. These extra features give additional repulsions between the cars and we cannot expect to have the same distribution as the examples above. Indeed this was confirmed for the parked cars in parking garages in Ann Arbor, USA \cite{Fader}. 
If we were then to model this situation using random matrix theory, it is natural to imagine that the parking meters and the markings on the floor serve as external potentials which are periodic with the period in the same scale as the 
spacing of the eigenvalues. 
Then the eigenvalues are governed in part by the sine kernel due to the random matrix repulsion and in part by the potential. 
The result~\eqref{eq:Rlimit2} with $\lambda=1$ gives us the exact relationship between these two parts. 
The question of exactly which potentials would indeed give rise to the presence of the parking meters or the markings on the floor is left to the reader's imagination.

\section{Example}\label{sec:example}

As an example, we consider the case when the potential of~\eqref{eq:density} is 
\begin{equation}
	V(e^{i\theta})= -t \cos(\theta), \qquad t>0.
\end{equation}
in more detail. 
It will be clear from the calculations below that the same analysis 
applies to more general potentials $V(e^{i\theta})= tW(e^{i\theta})$ where $W$ has the unique minimum in $[-\pi, \pi]$ and is locally quadradic near the minimizer.  
However, we state the results only for the above potential. 

This potential has the minimum at $\theta=0$ as a function $\theta\in [-\pi, \pi]$.  
Near the minimizer, the potential is locally quadratic.
Hence the potential $V_\Lambda(e^{i\theta})= -t\cos(\Lambda\theta)$ has $\Lambda$ minima at $\theta=0, \frac{2\pi}{\Lambda}, \frac{4\pi}{\Lambda}, \cdots, \frac{2(\Lambda-1)\pi}{\Lambda}$. 
In other words, the potential $V_\Lambda$ has $\Lambda$ wells.
When $t>0$ becomes larger, the well at each minimizer becomes deeper. 
We will evaluate the effect when $t\to \infty$ at the same time when $N\to \infty$. 
In order to make the presentation simple, we only consider the case when $\frac{\Lambda}{N}=\lambda$ is a fixed positive integer, i.e. the number of wells is an integer multiple of the number of eigenvalues. Then the reproducing kernel is simply given by~\eqref{eq:corsim}. 
Note that 
\begin{equation}
    w_0= \int_{-\pi}^{\pi} e^{t\cos\theta} d\theta
    \approx \sqrt{\frac{2\pi}{t}} e^{t}, \qquad \text{as $t\to\infty$.}
\end{equation}

We first consider the correlation functions. 
If we set $\xi_j=\frac{2\pi \ell_j}{\lambda}+ \frac{u_j}{\lambda \sqrt{t}}$ in~\eqref{eq:Rlimit2} where $\ell_j$ are integers and $u_j=O(1)$, then we have 
\begin{equation}
    \frac{w(e^{i\lambda \xi_1})\cdots w(e^{i \lambda \xi_m})}{w_0^{m}}
    \approx \bigg( \frac{t}{2\pi} \bigg)^{m/2} e^{-\frac12(u_1^2+\cdots + u_m^2)}
\end{equation}
and 
\begin{equation}\label{eq:Sellde}
    \det\bigg[ \frac{\sin(\frac12(\xi_a-\xi_b))}{\frac12(\xi_a-\xi_b)} \bigg]_{a,b=1}^{m}
    \to \det\bigg[ \frac{\sin(\frac{\pi}{\lambda}(\ell_a-\ell_b))}{\frac{\pi}{\lambda}(\ell_a-\ell_b)} \bigg]_{a,b=1}^{m}
    =: S_m^{(\lambda)}(\ell_1, \cdots, \ell_m)
\end{equation}
as $t\to \infty$. 
Using Corollary~\ref{cor:simn1}, it is easy to check that this limit holds even if we take $N\to \infty$ and $t\to\infty$ simultaneously. 
Note that $\frac{2\pi \ell_j}{\lambda N}$ denotes the location of one of $\lambda N=\Lambda$ minimizers of the potential $V_{\lambda N}(e^{i\theta})= -t\cos(\lambda N\theta)$.

\begin{cor}
For $u_1, \dots, u_\ell$ in a compact subset of $\mathbb{R}$ and for $\ell_j\in \Z$, 
\begin{equation}
      \lim_{N,t\to\infty} \frac1{(N\sqrt{t})^m} R_{m}^{N, \lambda N} \bigg( e^{i(\frac{2\pi \ell_1}{\lambda N}+\frac{u_1}{\lambda N\sqrt{t}})},
      \cdots,
      e^{i(\frac{2\pi \ell_m}{\lambda N}+\frac{u_m}{\lambda N\sqrt{t}})} \bigg)
      = \frac1{(2\pi)^{m/2}} e^{-\frac12(u_1^2+\cdots+u_m^2)} S_m^{(\lambda)}(\ell_1, \cdots, \ell_m)
\end{equation}
where $S_m^{(\lambda)}(\ell_1, \cdots, \ell_m)$ is defined in~\eqref{eq:Sellde}.  
\end{cor}

Hence as $t\to\infty$,  the $m$-point correlation function converges to the density of $m$ independent Gaussian variables times the multiplicative factor $S_m^{(\lambda)}(\ell_1, \cdots, \ell_m)$. 
Note that $S_m^{(\lambda)}(\ell_1, \cdots, \ell_m) \in [0,1]$. 
It equals $1$ if $\ell_i-\ell_j\in \lambda \Z\setminus\{0\}$ for all $i, j$. 
Thus, the multiplicative factor $S_m^{(\lambda)}(\ell_1, \cdots, \ell_m)$ is the largest when the distances between $\ell_j$ are integer multiples of $\lambda$. 
This coincides with the intuition since there are $\lambda$ times more minimizers than the eigenvalues. 

\bigskip

We next evaluate the probability that all eigenvalues are near the $\lambda N$ minimizers of $V_{\lambda N}(e^{i\theta})= -t\cos(\lambda N\theta)$. 
For $\epsilon\in (0,\pi)$, set 
\begin{equation}
  \Gamma_{\epsilon} 
  = \bigcup_{j=1}^{\lambda N} \bigg\{ \theta :  \frac{2\pi (j-1)}{\lambda N} +\frac{ \epsilon}{\lambda N} \le \theta\le \frac{2\pi j}{\lambda N} -\frac{\epsilon}{\lambda N} \bigg\}.
\end{equation}
This set is the complement of the union of the $\frac{2\epsilon}{\lambda N}$-neighborhoods of the minimizers. 
We expect that the probability that there is no eigenvalue in $\Gamma_{\epsilon}$ is close to $1$. 
This can be seen as follows. 
The gap probability is $\det(1-\calK_{N, \lambda N}1_{\Gamma_{\epsilon}})$.
Using the basic estimate $\det(1-A)-\det(1-B)\le \|A-B\|_1 e^{\|A\|_1+\|B\|_1+1}$ for trace class operators (where $\|\cdot \|_1$ stands for the trace norm), we have 
\begin{equation}
  	1- \Prob(\text{no eigenvalue in $\Gamma_\epsilon$})= 1-\det(1-\calK_{N, \lambda N}1_{\Gamma_{\epsilon}}) \le \tr(\calK_{N, \lambda N}1_{\Gamma_{\epsilon}}) e^{\tr(\calK_{N, \lambda N}1_{\Gamma_{\epsilon}})+1}.
\end{equation}
Here we used the fact that $\calK_{N, \lambda N}1_{\Gamma_{\epsilon}}$ is a positive trace-class operator, 
and therefore  $\|\calK_{N, \lambda N}1_{\Gamma_{\epsilon}}\|_1= \tr (\calK_{N, \lambda N}1_{\Gamma_{\epsilon}} )$.
Now
\begin{equation}
   \tr(\calK_{N, \lambda N}1_{\Gamma_{\epsilon}}1_{\Gamma_{\epsilon}})
   = \int_{\Gamma_{\epsilon}} \calK_{N, \lambda N}(e^{i \theta},e^{i \theta})d\theta
   = \frac{N}{w_0} \int_{\Gamma_{\epsilon}} w(e^{i \lambda N \theta}) d\theta
\end{equation}
using Theorem~\ref{thm1} with $\xi=\eta=N \theta$. 
Hence 
\begin{equation}
   \tr(\calK_{N, \lambda N}1_{\Gamma_{\epsilon}}1_{\Gamma_{\epsilon}})
   = \frac{N}{w_0} \int_{ \epsilon}^{2\pi- \epsilon}   e^{t\cos\phi} d\phi
   = O\bigg( \frac{Ne^{t\cos \epsilon}}{t w_0} \bigg)
   = O\bigg( \frac{Ne^{-\frac14\epsilon^2 t}}{\sqrt{t}} \bigg)
\end{equation}
and we find that $\Prob(\text{no eigenvalues in $\Gamma_{\epsilon}$})\to 1$ as $N, t\to\infty$ and, for example, $t\ge (\log N)^2$. 

\bigskip

We finally evaluate the probability that there is exactly one eigenvalue near one of the minimizers of $V_{\lambda N}(\theta)$. 
Since there are $\lambda N$ minimizers and $N$ eigenvalues, it is expected that this probability is close to $\frac1{\lambda}$.
By symmetry, we may consider the neighborhood of $\theta=0$. 
For $\epsilon\in (0,1/2)$, set 
\begin{equation}\label{eq:Jep}
	J_\epsilon :=  \big\{ \theta: -\frac{2\pi\epsilon}{\lambda N} \le \theta\le \frac{2\pi\epsilon}{\lambda N} \big\}.
\end{equation}
From the theory of determinantal point processes, the probability that there is precisely $k$ eigenvalues in an interval $I$ is given by $\frac{d^k}{dz^k}\big|_{z=-1} \det(1+z\calK_{N, \lambda N}1_{I})$. 
With $I=J_\epsilon$, from the definition of Fredholm determinant, 
\begin{equation}\label{eq:detzcalK}
	\det(1+z\calK_{N, \lambda N}1_{J_\epsilon}) = 1+zM_1+ \frac{z^2}{2!} M_2+ \frac{z^3}{3!} M_3+\cdots
\end{equation}
where 
\begin{equation}\label{eq:Mm}
\begin{split}
	M_m
	&= \int_{J_\epsilon}\cdots\int_{J_\epsilon} \det(\calK_{N, \lambda N}(e^{i\theta_j}, e^{i\theta_k}))_{j,k=1}^m d\theta_1\cdots d\theta_m \\
	&= \bigg( \frac{2\pi}{\lambda N}\bigg)^m 
	\int_{-\epsilon}^\epsilon \cdots\int_{-\epsilon}^\epsilon 
	\det\bigg[ \frac{\sin(\frac{\pi}{\lambda}(\xi_j-\xi_k))}{\sin(\frac{\pi}{\lambda N}(\xi_j-\xi_k))} \bigg]_{j,k=1}^m  \prod_{j=1}^m \frac{e^{t\cos(2\pi\xi_j)}}{w_0} d\xi_j
\end{split}
\end{equation}
using Corollary~\ref{cor:simn1} for the last equality. 

For $m=1$, 
\begin{equation}\label{eq:M1}
\begin{split}
	M_1
	&= \frac{2\pi}{\lambda} \int_{-\epsilon}^\epsilon   \frac{e^{t\cos(2\pi\xi)}}{w_0} d\xi
	= \frac{1}{\lambda} +O(e^{-ct})
\end{split}
\end{equation}
for some constant $c>0$.

Let us consider $m\ge 2$. As $N\to \infty$, the determinant in~\eqref{eq:Mm} times $N^{-m}$ converges to 
$S_m^{(\lambda)}(\xi_1, \cdots, \xi_m)$ defined in~\eqref{eq:Sellde}. 
Hence as $N\to \infty$, 
\begin{equation}\label{eq:Mm200}
\begin{split}
	M_m
	&\approx  \bigg( \frac{2\pi}{\lambda}\bigg)^m 
	\int_{-\epsilon}^\epsilon \cdots\int_{-\epsilon}^\epsilon 
	S_m^{(\lambda)}(\xi_1, \cdots, \xi_m) \prod_{j=1}^m \frac{e^{t\cos(2\pi\xi_j)}}{w_0} d\xi_j.
\end{split}
\end{equation}
We use the Laplace method to evaluate the integral asymptotically as $t\to \infty$. The main contribution to the integral is attained when $\xi_1=\cdots = \xi_m=0$. 
But the function $S_m^{(\lambda)}(\xi_1, \cdots, \xi_m)$ is $0$ when $\xi_1=\cdots= \xi_m=0$. 
Hence we need to expand this function using the Taylor series at the origin. In doing so, since $e^{t\cos(2\pi \xi)}$ is an even function, the leading term comes from a constant times 
\begin{equation}\label{eq:otherM1}
\begin{split}
	 \int_{-\epsilon}^\epsilon \cdots\int_{-\epsilon}^\epsilon 
	(\xi_1^2+\cdots + \xi_m^2)  \prod_{j=1}^m \frac{e^{t\cos(2\pi \xi_j)}}{w_0} d\xi_j
	\approx m  \frac{\int_{-\epsilon}^\epsilon 
	\xi^2 e^{- 2\pi^2 t \xi} d\xi }{\int_{-\epsilon}^\epsilon e^{- 2\pi^2 t \xi} d\xi} = O\big(\frac{m}{t} \big), 
	\quad m\ge 2.
\end{split}
\end{equation}
It is not difficult to see that this is uniform in all large enough $N$ and we obtain $M_m=O(t^{-1})$ for $m\ge 2$.

Since the matrix in~\eqref{eq:Mm} is a positive matrix,  the determinant is bounded by the product of the diagonal terms of the matrix from the Hadamard's inequality. Thus the determinant is bounded by $1$ and hence 
\begin{equation}
\begin{split}
	M_m
	&\le  \bigg( \frac{2\pi}{\lambda}\bigg)^m   \int_{-\epsilon}^\epsilon \cdots\int_{-\epsilon}^\epsilon   \prod_{j=1}^m \frac{e^{t\cos(2\pi \xi_j)}}{w_0} d\xi_j \le \frac{1}{\lambda^m}
\end{split}
\end{equation}
for all $m$. 
Therefore, from the dominated convergence theorem, 
the limits of~\eqref{eq:detzcalK} and also any derivatives as $N, t\to\infty$ can be evaluated by taking a term-by-term limit. 
Hence
\begin{equation}
\begin{split}
	\Prob(\text{ precisely 1 eigenvalue in $J_\epsilon$})
	= M_1- M_2+\frac1{2!} M_3 -\frac1{3!}M_4 +\cdots  \to \frac1{\lambda}
\end{split}
\end{equation}
as $N, t\to\infty$. 
A more precise estimate, which we do not discuss here, implies that the above probability is indeed $\frac1{\lambda}+O(t^{-1})$.  


Since there are $\lambda N$ minimizers of the potential $V_{\lambda N}(\theta)$, 
the above result implies that the probability that every eigenvalue is in the $\frac{4\pi \epsilon}{\lambda N}$-neighborhood of a minimizer and each neighborhood  is occupied by at most one eigenvalue 
is $1+O(\frac{N}{t})$. This tends to $1$ if $\frac{t}{N}\to\infty$ and $t, N\to \infty$. 

\subsubsection*{Acknowledgments}
I would like to thank Peter Forrester for bringing my attention to his paper \cite{Forrester} after the completion of the first version of this paper. 
This work was supported in part by NSF grants DMS1068646.


\end{document}